\documentclass[letterpaper,10pt,oneside]{article}

\def\TITLE{Varieties of modal algebras without the congruence extension property}
\def\AUTHOR{Zal\'an Gyenis\thanks{Jagiellonian University},\quad\quad Zal\'an Moln\'ar\thanks{E\"otv\"os Lor\'and Univeristy}}
\def\DATE{\today}
\def\ABSTRACT{In a recent paper, Krawczyk \cite{Krawczyk2023} proved that there
are continuum many axiomatic extensions of global consequence associated with the modal system $E$ that do not admit the local deduction detachment theorem. In algebraic parlance, he showed that there are continuum many varieties of modal algebras lacking the congruence extension property. In this paper, we extend Krawczyk's results and construct a continuum of varieties of modal algebras that do not have the congruence extension property, but that do admit other, logically relevant properties, such as monotonicity, extensiveness, idempotency, normality, etc. This gives a 
continuum of axiomatic extensions of the corresponding modal systems not having the local deduction detachment theorem.}
\def\KEYWORDS{Modal algebras, Congruential modal logics, Congruence extension property, Local deduction detachment theorem} 
\def\SUBJCLASS{Primary 03B45, 03G25; Secondary 03G27}   

\def\enddefsymbol{$\square$}				
\def\endproofsymbol{$\blacksquare$}			

\usepackage[utf8]{inputenc}
\usepackage[T1]{fontenc}
\usepackage[english]{babel}

\usepackage{amsmath,amssymb, graphicx, amsopn, amsthm}
\usepackage{enumerate, framed}
\usepackage[explicit]{titlesec}
\usepackage{nicefrac}

\usepackage{setspace}
\newcommand{\MAKETITLE}
{
\title{\textsc{\textbf{\MakeLowercase{\TITLE}}}}
\author{\AUTHOR}
\date{\DATE}
\maketitle
\begin{abstract}
	\noindent \ABSTRACT

	\ifx\KEYWORDS\empty\else
	  \vspace{5mm}
	  \noindent {\bf Keywords:} \KEYWORDS.
	\fi
	\ifx\SUBJCLASS\empty\else
	
	  \vspace{1mm}
	  \noindent {\bf Subject classification:} \SUBJCLASS.
	\fi
\end{abstract}

\vspace{5mm} \normalsize	
}

\usepackage[overload]{textcase} 	

\newcommand{\gyzFormatSec}[1]{\textbf{\textsc{\MakeLowercase{#1}}}}

\titleformat{\section}
{\normalfont}{\Large\fontfamily{yvtj}\selectfont\thesection}{1em}{\Large\gyzFormatSec{#1}}

\titleformat{\subsection}
{\normalfont}{\large\fontfamily{yvtj}\selectfont\thesubsection}{1em}{\large\gyzFormatSec{#1}}

\makeatletter
\let\osection\section
\def\section{\@tempskipa\lastskip\removelastskip
\penalty-20
\vskip\@tempskipa\osection}
\let\osubsection\subsection
\def\subsection{\@tempskipa\lastskip\removelastskip
\penalty-20
\vskip\@tempskipa\osubsection}
\makeatother

\newtheoremstyle{mythmstyle}
  {\topsep}
  {\topsep}
  {\normalfont}
  {}
  {}
  {\bfseries.}
  {.7em}
 {{{\bfseries\thmname{#1}}~{\bfseries\thmnumber{#2}}{\normalfont{\thmnote{ (#3)}}}}} 

\theoremstyle{mythmstyle} 
\newtheorem{theorem}{Theorem}[section]		
\newtheorem{lemma}[theorem]{Lemma}

\newtheorem{remark}[theorem]{Remark}

\newtheorem{proposition}[theorem]{Proposition}
\newtheorem{corollary}[theorem]{Corollary}

\newtheorem{construction}[theorem]{Construction}

\newtheorem{defi/}[theorem]{Definition}

\newtheoremstyle{mynotestyle}
  {\topsep}
  {\topsep}
  {\normalfont}
  {}
  {}
  {\bfseries.}
  {.7em}
  {{{\bfseries\thmnumber{#2}}{\normalfont{\thmnote{ (#3)}}}}}

\theoremstyle{mynotestyle}

\renewenvironment{proof}[1][\unskip]{%
\par
\noindent
\textbf{Proof #1.}
\noindent}
{\hfill\endproofsymbol

\bigskip}

\usepackage{relsize}

\def\NN{\mathbb{N}}
\def\D{\mathop{\Diamond}}
\def\B{\mathop{\Box}}
\def\0{0}
\def\1{1}

\def\Bl{\mathfrak{Bl}}
\def\VAR{\mathbf{V}}

\def\models{\vDash}
\def\nmodels{\nvDash}
\def\f{\mathop{f}}

\def\Bao{\mathsf{Bf}}
\def\LAND{\sqcap}
\def\LOR{\sqcup}

\def\UAR{\mathbf{U}}
\def\star{{}^{*}}

\newtheorem{manualtheoreminner}{Theorem}
\newenvironment{manualtheorem}[1]{%
  \renewcommand\themanualtheoreminner{Theorem #1}%
  \manualtheoreminner
}{\endmanualtheoreminner}

\def\bfS{\mathbf{S}}
\def\bfE{\mathbf{E}}
\def\bfK{\mathbf{K}}

\def\bfER{\mathbf{ER}}
\def\bfEC{\mathbf{EC}}

\def\bfENS{\mathbf{ENS}}

\def\bfENMRI{\mathbf{ENMRI}}

\let\<=\langle
\let\>=\rangle

\def\[#1\]{\begin{align}#1\end{align}}

\def\land{\wedge}       
\def\lor{\vee} 
\def\Land{\bigwedge}    

\def\lnot{\neg}

\let\models=\vDash

\def\proves{\vdash}

\renewcommand{\setminus}{\smallsetminus}

\def\defeq{\overset{\text{def}}{=}}

\DeclareMathOperator{\Hom}{Hom}

\DeclareMathOperator{\Con}{Con}

\def\HSP{\mathbf{HSP}}

\def\gA{\mathfrak{A}}
\def\gB{\mathfrak{B}}
\def\gC{\mathfrak{C}}

\def\gP{\mathfrak{P}}

\def\gW{\mathfrak{W}}

\def\cP{\mathcal{P}}

\def\cW{\mathcal{W}}

\usepackage{tikz-cd}
\usetikzlibrary{graphs, graphs.standard}

\usepackage{scalerel,stackengine}
\def\apeqA{\SavedStyle\sim}
\def\SIM{\setstackgap{L}{\dimexpr.5pt+0.5\LMpt}\ensurestackMath{%
  \ThisStyle{\mathrel{\Centerstack{{\apeqA} {\apeqA}}}}}}

\usepackage{stmaryrd}

\begin{document}
	\MAKETITLE

\section{Overview}

The standard propositional modal language is given by a countably infinite set
of propositional letters $P$ and the usual logical connectives $\land$, $\lnot$, $\bot$, and $\D$. Other connectives such as $\top$, $\B$, $\to$, $\lor$, etc. are taken as defined in the familiar ways. Let $\bfS$ be a modal system, i.e. a set of proof rules and axiom schemes. For a formula $\psi$ and a set of formulas $\Gamma$
we write $\Gamma\proves_{\bfS}\psi$ and say that $\psi$ can be derived from $\Gamma$, if there is a Hilbert-style derivation of $\psi$ that uses only the formulas of $\Gamma$, and the axioms and the proof rules of $\bfS$. 
The modal system $\bfE$ has any axiomatization of propositional calculus as its axiom schemes, and has the modus ponens (MP) and the congruentiality rule (RE) as its proof rules.\footnote{Sometimes axioms are given using propositional variables and a rule of substitution is given that permits to substitute any formula for them. From now on we shall instead regard axioms as schematic, so that substitution is superfluous and indeed admissible.}
$$
    (RE)\qquad \frac{\varphi\leftrightarrow \psi}{\B \varphi\leftrightarrow\B \psi}
    \qquad\qquad\qquad\qquad
    (MP)\qquad \frac{\varphi, \varphi\rightarrow\psi}{\psi}
$$
The set of $\proves_{\bfE}$-derivable formulas $E$ is referred to as the smallest \emph{classical modal logic} or \emph{congruential modal logic}, i.e. this is the smallest set of modal formulas closed under instances of propositional tautologies and the rules (RE) and (MP), cf. \cite{Pacuit2017}. 

The consequence relation $\proves_{\bfS}$ (resp. the system $\bfS$) is an \emph{axiomatic extension} of $\proves_{\bfE}$ (resp. the system $\bfE$), if there is a set of formulas $\Sigma$, such that 
$$
    \Gamma\proves_{\bfS}\psi
        \quad\text{ if and only if }\quad
    \Gamma\cup\Sigma \proves_{\bfE} \psi
$$
for any set $\Gamma\cup\{\psi\}$ of formulas. This means that $\proves_{\bfS}$ can be obtained by adding new axioms $\Sigma$ to the system $\bfE$.
By a \emph{Boolean frame} ($\Bao$, for short) we understand a structure $\gA = \<A, \LAND, -, \0, \f\>$, where $\<A, \LAND, -, \0\>$ is a Boolean algebra,\footnote{We make use of other standard Boolean operations such as $\LOR$, $\to$, $\1$. These are defined in the usual way.} and $\f:A\to A$ is an arbitrary unary function.\footnote{In parts of the literature a modal algebra is a Boolean algebra with an extra \emph{operator} $\D$ that is \emph{normal} and \emph{additive} (that is, $\D\0=\0$ and $\D(a\lor b) = \D a\lor \D b$; or, using the dual operator $\B$ we can write $\B\1=\1$ and $\B(a\land b)=\B a\land\B b$; see e.g. \cite[Def. 10.1]{Ono2019}). In some other parts of the literature, e.g. in \cite{Krawczyk2023}, a modal algebra
is just a Boolean algebra with an arbitrary extra unary function. To resolve this conflict we follow \cite{Hansson1973} and use the expression Boolean frame to refer to modal algebras in the latter sense.} The class $\Bao$ is an equational class (variety) defined by the Boolean equations. The congruentiality rule written in the quasi-equational form $x=y\rightarrow \f(x)=\f(y)$ holds in any Boolean frame. It is known that $\proves_{\bfE}$ is algebraizable in the Blok--Pigozzi \cite{BP89} sense by the class $\Bao$. This means the following.
Consider the propositional letters as algebraic variables, and define the translation $\iota$ between modal formulas and algebraic terms inductively by letting $\iota(p)=p$ for $p\in P$, and
$$
    \iota(\bot)=\0,  \quad
    \iota(\varphi\land\psi) = \iota(\varphi)\LAND\iota(\psi), \quad
    \iota(\lnot\varphi)= -\iota(\varphi), \quad
    \iota(\D\varphi) = f(\iota(\varphi)).    
$$
Then for any set of formulas $\Gamma\cup\{\psi\}$ we have
\[
    \Gamma\proves_{\bfE}\psi
        \quad\text{ if and only if }\quad
    \Bao\models \Land_{\gamma\in\Gamma}\iota(\gamma)=\1\  \rightarrow\ 
    \iota(\psi)=\1
\]
In the special case when $\Gamma$ is empty, we obtain
\[
    \proves_{\bfE}\psi\leftrightarrow\varphi
    \quad&\text{ if and only if}\quad
    \Bao\models \iota(\psi)=\iota(\varphi), \text{ and }\\
    \proves_{\bfE}\psi
    \quad&\text{ if and only if}\quad
    \Bao\models \iota(\psi)=\1.
\]
Algebraizability of $\proves_{\bfE}$ by a \emph{variety} also ensures that the lattice of axiomatic extensions of $\proves_{\bfE}$ is dually isomorphic to the lattice of subvarieties of $\Bao$ (see e.g. \cite{Krawczyk2023}). \\

A consequence relation $\proves$ has a \emph{local deduction detachment theorem} (LDDT) if for any $\Sigma\cup\{\varphi,\psi\}$ there is a finite set of formulas $L(p,q)$ such that 
$$
	\Sigma\cup\{\varphi\}\proves\psi \quad\text{ if and only if }\quad
	\Sigma\proves L(\varphi, \psi).
$$
An algebra $\gA$ has the \emph{congruence extension property} (CEP) if for any subalgebra $\gB\subseteq\gA$ and any congruence $\Theta\in\Con(\gB)$ there
is $\Psi\in\Con(\gA)$ such that $\Theta = \Psi\cap(B\times B)$. A class of algebras
has CEP if all its members have CEP. If the consequence relation $\proves$ is algebraizable by a variety $\mathsf{V}$, then the local deduction theorem for axiomatic extensions of $\proves$ corresponds to the congruence extension property of the subvarieties of $\mathsf{V}$ (see Czelakowski \cite{Czelakowski1986}). In particular, an axiomatic extension of $\proves_{\bfE}$ has a LDDT if and only if
the corresponding subvariety of $\Bao$ has the CEP. \\

The modal logic $K$ is the smallest set of formulas closed under instances
of propositional tautologies, the axiom scheme (K), and the rules (MP) and (Nec).
$$
    (K)\qquad \B(\varphi\to\psi)\to(\B\varphi\to\B\psi)
    \qquad\qquad\qquad\quad
    (Nec)\qquad \frac{\varphi}{\B\varphi}
$$
The system $\bfK$ is defined analogously. 
$K$ is referred to as the smallest \emph{normal modal logic}. It can be proved
that $\proves_{\bfE}$ is strictly weaker than $\proves_{\bfK}$, and that
$\proves_{\bfK}$ is an axiomatic extension of $\proves_{\bfE}$ by adding
the following axioms to $\bfE$:
$$
    (M)\quad \B(\varphi\land\psi)\to(\B\varphi\land\B\psi)
    \qquad
    (C)\quad (\B\varphi\land\B\psi)\to\B(\varphi\land\psi)
    \qquad
    (N) \quad \B\top
$$
(See Corollary 2.1 in \cite{Pacuit2017}). Even though $\bfK$ does not have
a deduction theorem \cite{Czelakowski1985}, it has a local deduction detachment theorem (Perzanowski, cf. \cite{Czelakowski1986}). Moreover,
every axiomatic extension of $\proves_{\bfK}$ has a LDDT, equivalently, every \emph{normal and additive} subvariety of $\Bao$ has the CEP.\footnote{This statement is part of the folklore, mentioned by e.g. \cite[Proposition 1]{Kowalski2000}. For completeness, we give a short proof of this statement in the Appendix, and in a subsequent proposition, we also show that normality is not needed for the result.}\\

In a recent paper, Krawczyk \cite{Krawczyk2023} proved that there
are continuum many axiomatic extensions of $\proves_{\bfE}$ that do not admit the local deduction detachment theorem. In algebraic parlance, he showed that there are continuum many subvarieties of $\Bao$ that lack the congruence extension property (see \cite[Theorem 6.11]{Krawczyk2023}). 
These axiomatic extensions of $\proves_{\bfE}$ are rather ad hoc and lack any of the properties often considered in the literature, such as transitivity, reflexivity, normality, etc. Algebraically,
the $\Bao$-subvarieties constructed in \cite{Krawczyk2023} lack the properties of being idempotent, monotone, normal, etc. It remained open whether there are  monotonic, normal, etc. axiomatic extensions of $\proves_{\bfE}$ without the local deduction detachment theorem. This paper settles this question, improving on the results of \cite{Krawczyk2023}. \\

\noindent The results of the paper are listed below. All the constructions and proofs are in section \ref{sec:constructions}. We list here the main axioms and their algebraic equivalents that we use in the paper.


\begin{center}
\begin{tabular}{c|l||l|l}
    name & axiom                        & algebraic equation  & name \\ \hline
    (N)  & $\D\bot\leftrightarrow\bot$  
            & $\f(\0) = \0$ & normal \\
    --  & $\D\top\leftrightarrow\top$ 
            & $\f(\1) = \1$ & unit-preserving \\
    (K)  & $\D(\varphi\lor\psi)\leftrightarrow (\D\varphi\lor\D\psi)$ 
            &    $\f(x\LOR y) = \f(x)\LOR \f(y)$ & additive \\
    (R)  & $\varphi\to\D\varphi$ 
            & $x\leq \f(x)$ & extensive \\
    --  & $\D\varphi\to\phi$ 
            & $\f(x)\leq x$ & contractive \\            
    (I)  & $\D\D\varphi\leftrightarrow\D\varphi$ 
            & $\f(\f(x)) = \f(x)$ & idempotent \\
    (S) & $\D\lnot\varphi\leftrightarrow\lnot\D\varphi$ 
            & $\f(-x) = -\f(x)$ & semi-complemented \\
    (M) & $(\D\varphi\lor\D\psi)\to \D(\varphi\lor\psi)$ 
            & $\f(x)\LOR\f(y)\leq \f(x\LOR y)$ & monotone \\
    (C)  & $\D(\varphi\lor\psi)\rightarrow (\D\varphi\lor\D\psi)$ 
            &    $\f(x\LOR y) \leq \f(x)\LOR \f(y)$ & subadditive  \\         
    (B) & $\varphi\to\B\D\varphi$
            & $x\leq -\f(-\f(x))$ & symmetric
\end{tabular}
\end{center}

\noindent Note that the equation corresponding to monotonicity is
indeed equivalent to the quasi-equation $x\leq y \ \Rightarrow\ \f(x)\leq \f(y)$. It is straightforward that (M) and (C) together are equivalent to additivity (K). The axiom (K) is standardly given by the formula
$\B(\varphi\to\psi)\to(\B\varphi\to\B\psi)$. This formulation and the one in the table above are $\proves_{\bfE}$-provably equivalent, see \cite[Lemma 2.8 and Exercise 2.16]{Pacuit2017}. We chose the formulation that is more similar to the corresponding algebraic form.\\


\bigskip

\begin{center} {\bf The results of the paper}  \end{center}

\noindent {\bf \#1.}
The logic $\bfER$ extends $\bfE$ by the axiom (R).
The consequence relation $\proves_{\bfER}$ is algebraized by the variety of extensive Boolean frames. As a warming up, in subsection \ref{subsec:extensive}, we give a \emph{simple} construction and proof for the following theorem.

\begin{manualtheorem}{\ref{THM:EXT}}
    There are continuum many \emph{extensive} subvarieties of Boolean frames
	lacking the congruence extension property.
\end{manualtheorem}

\noindent Theorem \ref{THM:EXT} yields:

\begin{corollary}
    $\proves_{\bfER}$ has continuum many axiomatic extensions 
    that do not admit the local deduction detachment theorem. 
\end{corollary}

To prove Theorem \ref{THM:EXT}, for each subset $X\subseteq \mathbb{N}$
we construct a Boolean frame $\gA_X$ lacking the CEP, in such a way that for $X\neq Y$ the varieties generated by $\gA_X$ and $\gA_Y$ are distinct, shown by equalities.
For each $X\subseteq\NN$ the Boolean frame $\gA_X$ is based on the powerset Boolean algebra $\gP(\NN)$ (cf. Construction \ref{constr:1}). It follows that for each such $X$ there is a neighborhood frame $\<\NN, N_{X}\>$ such that the logic of $\gA_X$ (and hence the logic of the variety generated by $\gA_X$) is the same as the logic of the neighborhood frame $\<\NN, N_{X}\>$.

By turning the algebras $\gA_X$ ``upside-down'' (Construction \ref{constr:contractive}), we get continuum many \emph{contractive} varieties of Boolean frames without the congruence extension property (Theorem \ref{THM:CONT}). This corresponds to continuum many axiomatic extensions of $\bfE + (P)$ that do not admit the local deduction theorem.\\

A small modification in Construction \ref{constr:1} leads to Construction \ref{constr:subadd} in subsection \ref{subsec:subadd}, and yields Theorem \ref{THM:SUBADD} that states that there are continuum many \emph{subadditive} varieties of Boolean frames	lacking the congruence extension property. 
The logic $\bfEC$ extends $\bfE$ by the axiom (C). The consequence relation $\proves_{\bfEC}$ is algebraized by the variety of subadditive Boolean frames.
Theorem \ref{THM:SUBADD} immediately gives us:

\begin{corollary}
    $\proves_{\bfEC}$ has continuum many axiomatic extensions 
    that do not admit the local deduction detachment theorem. 
\end{corollary}
\medskip


\noindent {\bf \#2.} We turn to extensions of $\bfE$ by the axioms (N), (R), (M), and (I). Our main theorem here is:
\begin{manualtheorem}{\ref{thm:main2}}
    For each variety $\VAR$ of \emph{normal} and \emph{unit-preserving} Boolean frames there exists a variety $\VAR\star$ of normal and unit-preserving Boolean frames such that the following properties hold.
    \begin{itemize}\itemsep-2pt
        \item $\VAR\star$ lacks the congruence extension property.
        \item If $\VAR\neq\UAR$, then $\VAR\star\neq\UAR\star$.
        \item The construction $\VAR\mapsto\VAR\star$ preserves the properties of being (not) extensive, (not) monotone, (not) idempotent.\footnote{E.g. if $\VAR$ is monotone but not extensive, then so is $\VAR\star$. We note that $\VAR\star$ is never additive. This is immediate from the construction and also follows from the fact that additive varieties have the congruence extension property, while $\VAR\star$ does not have it. }
    \end{itemize}
\end{manualtheorem}
Further, it will be clear from the construction that if $\VAR$ is generated by a finite set of finite algebras, then so is true for $\VAR\star$. In particular, if $\VAR$ is generated by a single (finite) algebra, then so is $\VAR\star$. Thus, tabularity, finite approximability of $\VAR$ is inherited to $\VAR\star$ (cf. Remark \ref{remark:tab}). \\

\noindent Theorem \ref{thm:main2} gives the following type of results: 
\begin{center}
    If there are $\kappa$ varieties of normal, unit-preserving 
    
    Boolean frames having property P

    $\Downarrow$

    there are $\kappa$ many such varieties without the CEP,
\end{center}
where $P$ is any combination of the properties of being (not) extensive, (not) monotone, (not) idempotent. \\

A normal, extensive, idempotent, and monotone operator is called a closure operator. The consequence relation $\proves_{\bfENMRI}$ is algebraized by the variety of Boolean frames with a closure operator. By applying Theorem \ref{thm:main2}, we get the following result.

\begin{corollary}\label{cor:corcor}
	There are continuum many \emph{normal, extensive, idempotent and monotone} varieties of Boolean frames lacking the congruence extension property.
    Therefore, there are continuum many axiomatic extensions of $\proves_{\bfENMRI}$ that do not admit a LDDT.
\end{corollary}

Corollary \ref{cor:corcor} follows from Theorem \ref{thm:main2} once one can construct continuum many varieties of Boolean frames with a closure operator. This has essentially been done e.g. in Fine \cite{Fine1974} disguised as the statement that there exists a continuum of normal modal logics containing the modal logic $S4$. \\

We also note that if $\gA = (A, f)$ is monotone, then $\gA' = (A, -f)$ is antitone, and as the two operations are term-interdefinable, the two algebras have the same subalgebras and congruences. It follows that there are continuum many varieties of \emph{antitone} Boolean frames without the congruence extension property.\\


\noindent {\bf \#3.} $\bfENS$ is the axiomatic extension of $\bfE$ by adding the axioms (N) and (S). The corresponding consequence relation $\proves_{\bfENS}$ is algebraized by the variety of normal and semi-complemented Boolean frames. In subsection \ref{subsec:sc} we prove the theorem below.

\begin{manualtheorem}{\ref{THM:COMP}}
    There are continuum many \emph{normal, unit-preserving and semi-complemented} varieties of Boolean frames lacking the congruence extension property.
\end{manualtheorem}

\noindent Theorem \ref{THM:COMP} immediately gives the following corollary:

\begin{corollary}
    There are continuum many axiomatic extensions of $\proves_{\bfENS}$ that do not admit a LDDT.
\end{corollary}

\noindent We note that there are only \emph{countably} many \emph{normal modal logics} that contain the axiom (S), see \cite[Exercise 6.22, p. 185]{chagrov1997modal}. It follows that there are only countably many varieties of normal, \emph{additive} and semi-complemented Boolean frames. \\


\def\bfENBR{\mathbf{ENBR}}

\noindent {\bf \#4.} $\bfENBR$ is the axiomatic extension of $\bfE$ by adding the axioms (N), (B), and (R). The corresponding consequence relation
is algebraized by the variety of normal, symmetric, and extensive Boolean frames. In subsection \ref{subsec:symext} we prove:

\begin{manualtheorem}{\ref{thm:main3}}
    There are continuum many \emph{normal, unit-preserving, extensive and symmetric} varieties of Boolean frames lacking the congruence extension property.
\end{manualtheorem}

\noindent An immediate corollary is:
\begin{corollary}
    There are continuum many axiomatic extensions of $\proves_{\bfENBR}$ that do not admit a LDDT.
\end{corollary}

The construction $(\cdot)^\flat$ and the subsequent proofs given in subsection \ref{subsec:symext} are combinations of the techniques from subsections \ref{subsec:sc} and \ref{subsec:main}. In particular, $(\cdot)^\flat$ is a construction that preserves symmetry and extensiveness. To obtain Theorem \ref{thm:main3} we make use of the symmetric and extensive wheel frames introduced by \cite{Miyazaki2005} and recalled in subsection \ref{subsec:sc}. 

\section{Constructions and proofs}\label{sec:constructions}

This section contains the constructions and proofs. Our notation is 
standard.


\subsection{Extensive and contractive varieties without the CEP}\label{subsec:extensive}

In this subsection our goal is to prove that there are continuum many \emph{extensive}, and continuum many \emph{contractive} subvarieties of Boolean frames	lacking the congruence extension property (Theorems \ref{THM:EXT} and \ref{THM:CONT}). 

\begin{construction}\label{constr:1}
    For each $X\subseteq\NN$ we define the the operation $\f$ on the
    powerset Boolean algebra $\gP(\NN)$ to obtain the Boolean frame
    $\gA_X = \<\gP(\NN), \f\>$.
    Write \[ E \defeq\{2n:\; n\in \NN\},\quad\text{ and }\qquad 2E \defeq\{2n:\;n\in E\}\,. \]
    For $S\subseteq\NN$ we let
    \[
    \f(S)\defeq\begin{cases}
        \{0,\ldots,n\} &\text{if } S=\{0,\ldots,n-1\}\\
        \NN&\text{if } S=2E\\
        \NN&\text{if } S=-\{0,\ldots,n\} \text{ and } n\in X\\
        S&\text{otherwise}
    \end{cases}
    \]
\end{construction}

\noindent It is straightforward from the definition that $\gA_X$ is extensive, that is, $\gA_X\models x\leq\f(x)$.

\begin{theorem}\label{EXT:2}
    $\gA_X$ does not have the congruence extension property.
\end{theorem}
\begin{proof}
    Let $\gB$ be the subalgebra of $\gA_X$ generated by $E$. Elements of $\gB$
    are all finite or co-finite in $\NN$; or finite or co-finite in $E$. In particular, $2E\notin B$. Observe that $\gB$ has a non-trivial congruence $\sim$ such that $E\sim\emptyset$. In fact, 
    let $\sim$ be the congruence generated by the filter
    \[\label{cong}
        F=\{ X\in B:\; X\text{ contains infinitely many odd numbers}\}.    
    \]
    We claim that $\sim$
    is not a restriction of any congruence of $\gA_X$. Indeed, let $\SIM$ be an extension of $\sim$ to a congruence of $\gA_X$. Then 
    \[
        E\sim\emptyset \quad\Rightarrow\quad 
        E\SIM\emptyset \quad\Rightarrow\quad
        2E\SIM\emptyset \quad\Rightarrow\quad 
        \f(2E)\SIM\f(\emptyset)    
    \]
    As $\f(2E)=\NN$ and $\f(\emptyset)=\{\emptyset\}$ it follows that
    $\NN\SIM\{\emptyset\}$. But $\{\emptyset\}\sim\emptyset$, therefore
    $\NN\SIM\emptyset$, and thus $\SIM$ is a trivial congruence.
\end{proof}

\begin{theorem}\label{EXT:3}\label{THM:EXT}
    For $X\neq Y\subseteq\NN$ we have $\HSP(\gA_X)\neq\HSP(\gA_Y)$.
    In particular, there are continuum many \emph{extensive} varieties of Boolean frames
	lacking the congruence extension property.
\end{theorem}
\begin{proof}
    Denoting the bottom and top elements respectively by $\0$ and $\1$, 
    it is straightforward to check that $\gA_X\models \f(-\f\nolimits^n(\0))=\1$ if and only if $n\in X$.
\end{proof}

Recall that a Boolean frame is contractive if the identity $\f(x)\leq x$ holds. To construct continuum many contractive subvarieties of Boolean frames lacking the congruence property one can simply turn ``upside down'' the construction. In more detail:

\begin{construction}\label{constr:contractive}
    For each $X\subseteq\NN$ we let
    $\gB_X = \<\gP(\NN), h\>$,
    where 
    \[
    h(S)\defeq\begin{cases}
        -\{0,\dots, n\} &\text{if } S=-\{0,\ldots,n-1\}\\
        \NN&\text{if } S=2E\\
        \emptyset &\text{if } S=\{0,\ldots,n\} \text{ and } n\in X\\
        S&\text{otherwise}
    \end{cases}
    \]
\end{construction}
\noindent Then $\gB_X$ is contractive. That $\gB_X$ does not have the congruence extension property follows from the same reasoning as in Theorem \ref{EXT:2}. Therefore, we have:
\begin{theorem}\label{THM:CONT}
    For $X\neq Y\subseteq\NN$ we have $\HSP(\gB_X)\neq\HSP(\gB_Y)$.
    In particular, there are continuum many \emph{contractive} varieties of Boolean frames
	lacking the congruence extension property.
\end{theorem}
\begin{proof}
 It is straightforward to check that $\gB_X\models  h(-h^n(\1))=\0$ if and only if $n\in X$.
\end{proof}


\subsection{The subadditive case}\label{subsec:subadd}

 A Boolean frame is called \emph{subadditive} if the identity $f(x\LOR y)\leq f(x)\LOR f(y)$ holds. In this subsection, we construct continuum many subadditive subvarieties of Boolean frames	lacking the congruence extension property (Theorem \ref{THM:SUBADD}).  For $n\in \NN$ let us introduce the notation $\overline{n} = \{n\}$.

\begin{construction}\label{constr:subadd}
	For each $X\subseteq\NN$ we define the  operation $g$ on the
	powerset Boolean algebra $\gP(\NN)$ to obtain the Boolean frame
	$\gC_X = \<\gP(\NN), g\>$.
	Write 
    \begin{align}
        E \defeq\{2n:\; n\in \NN\},\quad&\text{ and }\qquad 2E \defeq\{2n:\;n\in E\}\,, \\
         O \defeq\{2n+1:\; n\in \NN\},\quad&\text{ and }\qquad E^* \defeq\{S\in  [\NN]^\omega: |S\cap O| <\aleph_0, |E\setminus S| < \aleph_0 \}\,.
    \end{align}
  	For $S\subseteq\NN$, we let
	\[
	g(S)\defeq\begin{cases}
		\{0,\ldots,n\} &\text{if } S=\{0,\ldots,n-1\}\\
		S&\text{if } S\in E^*\\
		-\overline{n}&\text{if } S=-\overline{n} \text{ and } n\in X\\
		S\cup \{\max(S)+ 1\} &\text{if } S\in [\NN]^{<\omega} \text{ and } S\neq \{0,\dots, n-1\} \text{ for all $n\in \NN$} \\
		\NN&\text{otherwise}
	\end{cases}
	\]
\end{construction}

\begin{proposition}
	$\gC_X$ is subadditive for every $X\subseteq \NN$.
\end{proposition}
\begin{proof}
Pick $x,y\in \gC_X$. We prove the statement by a careful case selection. 
	
\noindent \textsc{Case 1:} Both $x,y\in [\NN]^{<\omega}$. This case is straightforward using the first and fourth lines of the definition of $g$. 

\noindent \textsc{Case 2:} $x\in [\NN]^\omega$ and $y\in [\NN]^{<\omega}$ (or vice versa).
	 \begin{enumerate}
	 	\item[]	\textsc{Subcase (a)}:
	 	 $-(x\LOR y)=\overline{n}$ such that $n\in X$. Then $g(x\LOR y) = x\LOR y$. Since for every $z\in \cP(\NN)$ we have $z\leq g(z)$, we obtain $g(x\LOR y) \leq g(x) \LOR g(y)$. 
	 	\item[]   \textsc{Subcase (b):} $-(x\LOR y)\neq \overline{n}$ for all $n\in X$. Since $x\LOR y\in [\NN]^\omega$ we have:
	 	\begin{enumerate}\itemsep-2pt
	 		\item[(i)] either $x\LOR y \in E^*$,
	 		\item[(ii)] or $x\LOR y \not\in E^*$.
	 	\end{enumerate} 
	 	For (i) we again have $g(x\LOR y) = x\LOR y$, hence similarly as above we get $g(x\LOR y) \leq g(x)\LOR g(y)$.  For (ii), by assumption we must have $x\in [\NN]^\omega$ and $x\not\in E^*$. Now assume that $x\LOR y$ is a co-atom. Hence $-(x\LOR y)=\overline{n}$ for some $n\not\in X$. If $y=0$, then also $-x = \overline{n}$, therefore $g(x) = \NN$, since $n\not\in X$, hence $g(x\LOR 	0)\leq g(x)\LOR g(0)$. If $y\neq 0$, then $x$ is not a co-atom, hence $g(x) = \NN$, therefore $g(x\LOR 	y)\leq g(x)\LOR g(y)$. Finally, if $x\LOR y$ is not a co-atom, then $x$ is also not a co-atom, but $g(x)= \NN$, since $x\not\in E^*$, and $x\in [\NN]^\omega$.
	 \end{enumerate}
\noindent \textsc{Case 3:} Both $x,y\in [\NN]^\omega$. 
\begin{enumerate}
	\item[] \textsc{Subcase (a):} $-(x\LOR y)=\overline{n}$ such that $n\in X$. Then $g(x\LOR y) = x\LOR y$, and at least one, say $x\not\in E^*$. But then $g(x) =\NN$, as $x\in[\NN]^\omega$, hence $g(x\LOR y)\leq g(x)\LOR g(y)$.
		\item[]   \textsc{Subcase (b):} $-(x\LOR y)\neq \overline{n}$ for all $n\in X$. This is done similarly to the corresponding part of \textsc{Case 2}.
\end{enumerate}
\end{proof}

\begin{theorem}\label{subadd_nocep}
	$\gC_X$ does not have the congruence extension property.
\end{theorem}
\begin{proof} Similarly to Theorem \ref{EXT:2} let $\gA$ be the subalgebra of $\gC_X$ generated by $E$. Observe that $\gA$ has a non-trivial congruence $\sim$ such that $E\sim\emptyset$ just as in (\ref{cong}). Now we can simply repeat the arguments from Theorem \ref{EXT:2}, we skip the details.
\end{proof}

\begin{theorem}\label{THM:SUBADD}
	For $X\neq Y\subseteq\NN$ we have $\HSP(\gC_X)\neq\HSP(\gC_Y)$.
	In particular, there are continuum many \emph{subadditive} varieties of Boolean frames
	lacking the congruence extension property.
\end{theorem}
\begin{proof} Observe that  in each $\gC_X$ every $\overline{n}$ is term definable by:
	\begin{itemize}\itemsep-2pt
		\item $\overline{0} = g(\emptyset)$, 
		\item  $\overline{n+1} = \bigsqcap_{m\leq n} -\overline{m}\LAND g^{n+2}(\emptyset)$.
	\end{itemize}
	Hence, $\gC_X\models g(-\overline{n}) = -\overline{n}$ if and only if $n\in X$.
\end{proof}


\subsection{Monotone, extensive and idempotent}\label{subsec:main}

In this subsection, we prove Theorems \ref{thm:main1} and \ref{thm:main2} which are used to construct normal, extensive, monotone, and idempotent varieties of Boolean frames lacking the congruence extension property. 

\begin{construction}\label{constr:acsillag}
    For an arbitrary Boolean frame $\gA = \<A, \LAND, -, \0, \f\>$
    we construct the Boolean frame $\gA\star$ as follows. 
    \begin{itemize}\itemsep-2pt
        \item The universe of $\gA\star$ is $A\star\defeq A\times A$.
        \item The Boolean reduct of $\gA\star$ is the direct product $\Bl\gA\times\Bl\gA$ of the Boolean reduct of $\gA$.
        \item The operation $\f\!\star:A\times A\to A\times A$ is defined as
        \[
            \f\!\star(\<a,b\>) \defeq\begin{cases}
                \< f(a),  f(b)\> & \text{ if } a=\0\text{ or }b=\0,\\
                \<\1,\1\> & \text{ otherwise.}
            \end{cases}
        \]
    \end{itemize}
    It is clear that $\gA\star$ is a Boolean frame.
\end{construction}

\noindent Recall that a Boolean frame $\gA=\<A,\LAND, -, \0, \f\>$ is \emph{normal} if
$\f(\0) =\0$,  and \emph{unit-preserving} if  $\f(\1)=\1$.

\begin{lemma}\label{lem:simple}
    Let $\gA= \<A, \LAND, -, \0, \f\>$ be a normal Boolean frame
    and suppose that there are $a, b\neq \0$ such that $\f(a)\LAND\f(b)=\0$. 
    Then $\gA\star$ is simple.
\end{lemma}
\begin{proof}
    Take any non-trivial congruence $\sim$ of $\gA\star$. Then there is
    $\<x,y\>\in A\star$ such that $\<x,y\>\neq\<\0,\0\>$ but 
    $\<x,y\>\sim\<\0,\0\>$. Without loss of generality we may assume $x\neq\0$, and thus $\<x,\0\>\sim\<\0,\0\>$ (because $\<x,\0\> \leq \<x,y\>$, and elements congruent to zero form an ideal). Let $a,b\in A$ be as in the statement of the lemma,  and  consider the elements 
    $\<x,a\>$ and $\<x,b\>$. Since $\<x,a\>\sim\<\0,a\>$ and $\<x,b\>\sim\<\0,b\>$ we  get
    \[
    \<\1,\1\> \ =\  \f\!\star(\<x,a\>)\ \sim\ \f\!\star(\<\0,a\>)\  =\  \<\0, \f(a)\>,\\
    \<\1,\1\> \ =\  \f\!\star(\<x,b\>)\ \sim\ \f\!\star(\<\0,b\>) \ =\  \<\0, \f(b)\>.
    \]
    But then 
    \[
        \<\1,\1\>\LAND\<\1,\1\> &\ =\  \f\!\star(\<x,a\>)\LAND\f\!\star(\<x,a\>)\\
         &\ \sim\  \f\!\star(\<\0,a\>)\LAND\f\!\star(\<\0,b\>)\\
         &\ = \ \<\0, \f(a)\>\LAND\<\0, \f(b)\> \\
         &\ =\ \<\0, \f(a)\LAND\f(b)\> \\
         &\ =\ \<\0,\0\>.
    \]
    Therefore, $\<\1,\1\>\sim\<\0,\0\>$, implying that $\sim$ is a trivial congruence.
\end{proof}

\begin{lemma}\label{lem:nocep}
    Let $\gA= \<A, \LAND, -, \0, \f\>$ be a normal and unit-preserving Boolean frame
    and suppose that there are $a, b\neq \0$ such that $\f(a)\LAND\f(b)=\0$. 
    Then $\gA\star$ does not have the  congruence extension  property.
\end{lemma}
\begin{proof}
    By normality and unitarity, the set
    $\big\{ \<\0,\0\>, \<\0,\1\>, \<\1,\0\>, \<\1,\1\> \big\}$
    is the universe  of a  subalgebra $\gB$ of $\gA\star$. This $\gB$ has 
    two non-trivial congruences as depicted below (the arrows illustrate the action of the operation $\f\!\star$, and the dotted bubbles are the congruence classes).
    \begin{figure}[!ht]
    \begin{center}
        \begin{tikzpicture}[thick,scale=1]
            \draw (0,0) node [below] {} -- 
                  (-1,1) node [left] {} --
                  (0,2) node [above] {} --
                  (1,1) node [right] {} -- (0,0); 
            \draw [fill] (0,0) circle [radius=.06]; 
            \draw [fill] (-1,1) circle [radius=.06];
            \draw [fill] (0,2) circle [radius=.06]; 
            \draw [fill] (1,1) circle [radius=.06];
            \draw[->,out=200,in=-30,looseness=10,loop,min distance=10mm,thin] (-0.1,0) to (0.1,0); 
            \draw[->,out=-110,in=110,looseness=10,loop,min distance=10mm,thin] (-1,1-0.1) to (-1,1+0.1); 
            \draw[->,out=40,in=150,looseness=5,loop,min distance=10mm,thin] (0.1,2) to (-0.1,2); 
            \draw[->,out=-70,in=70,looseness=5,loop,min distance=10mm,thin] (1,1-0.1) to (1,1+0.1); 

            \draw (0+5,0) node [below] {} -- 
                  (-1+5,1) node [left] {} --
                  (0+5,2) node [above] {} --
                  (1+5,1) node [right] {} -- (0+5,0); 
            \draw [fill] (0+5,0) circle [radius=.06]; 
            \draw [fill] (-1+5,1) circle [radius=.06];
            \draw [fill] (0+5,2) circle [radius=.06]; 
            \draw [fill] (1+5,1) circle [radius=.06];
            \draw[->,out=200,in=-30,looseness=10,loop,min distance=10mm,thin] (-0.1+5,0) to (0.1+5,0); 
            \draw[->,out=-110,in=110,looseness=10,loop,min distance=10mm,thin] (-1+5,1-0.1) to (-1+5,1+0.1); 
            \draw[->,out=40,in=150,looseness=5,loop,min distance=10mm,thin] (0.1+5,2) to (-0.1+5,2); 
            \draw[->,out=-70,in=70,looseness=5,loop,min distance=10mm,thin] (1+5,1-0.1) to (1+5,1+0.1);

            \begin{scope}[shift = {(0.5, 0.5)} ,rotate=45]
                \draw[dotted] (0, 0) ellipse (1.4 and 0.5);
            \end{scope}

            \begin{scope}[shift = {(0.5-1, 0.5+1)} ,rotate=45]
                \draw[dotted] (0, 0) ellipse (1.4 and 0.5);
            \end{scope}

            \begin{scope}[shift = {(0.5+4, 0.5)} ,rotate=135]
                \draw[dotted] (0, 0) ellipse (1.4 and 0.5);
            \end{scope}

            \begin{scope}[shift = {(0.5+5, 0.5+1)} ,rotate=135]
                \draw[dotted] (0, 0) ellipse (1.4 and 0.5);
            \end{scope}

\end{tikzpicture}
\caption{Non-trivial congruences of $\gB$.}
\label{figure:kisba}
    \end{center}
\end{figure}
    By Lemma \ref{lem:simple}, $\gA\star$ is simple, consequently, the non-trivial congruences of $\gB$  extend to trivial congruences of  $\gA\star$.    
\end{proof}

\noindent For a term $t(\vec x)$ and a variable $y$ not occurring in $\vec x$ we define
the relativized term $t^y(\vec x)$ by induction as follows.
\[ \0^y\defeq \0,\quad
x^y \defeq x\LAND y, \quad
(t_1\LAND t_2)^y\defeq t_1^y\LAND t_2^y,\quad
(-t)^y\defeq -(t^y)\LAND y,\quad
\f(t)^y\defeq \f(t^y) \label{relterm}
\]
Similarly, for an identity $t_1=t_2$ we let $(t_1=t_2)^y \defeq t_1^y=t_2^y$.\\

\noindent Observe that in a \emph{normal} Boolean frame $\gA$ we have
\[
\gA\models t_1=t_2 \quad \text{ if and only if }\quad 
\gA\star\models t_1^{\<\1,\0\>} = t_2^{\<\1,\0\>}\,,    \label{eq:erelativized}
\]
for any identity $t_1=t_2$ (and similarly  with  $\<\0,\1\>$ in place of $\<\1,\0\>$).\medskip

\def\UAR{\mathbf{U}}
\begin{theorem}\label{thm:main1}
    For a variety $\VAR$ of Boolean frames let us write
    \[ \VAR\star = \mathbf{HSP}\big\{ \gA\star:\; \gA\in \VAR \big\}.  \]
    Let $\VAR$ and $\UAR$ be different varieties of \emph{normal} and
    \emph{unit-preserving} Boolean frames. Then $\VAR\star$ and $\UAR\star$
    are different varieties lacking the congruence extension property.
\end{theorem}
\begin{proof}
    Take any $\gA\in \VAR$. Then $\gA\times\gA\in \VAR$ as well. In $\gA\times
    \gA$ there are non-zero elements $a$ and $b$ such that $\f(a)\LAND \f(b)=\0$:  take, for instance, $a = \<\1,\0\>$  and $b=\<\0,\1\>$. By Lemma
    \ref{lem:nocep},  the algebra $(\gA\times\gA)\star$ does not  have the
    congruence extension property. Clearly, $(\gA\times\gA)\star\in\VAR\star$, 
    therefore the  variety $\VAR\star$ lacks the congruence extension property as well. 

    Next, take any normal and unit-preserving Boolean frame $\gA$ and
    observe that in $\gA\star$ the only elements $a$ such that 
    $\f\!\star(a) = a$ \emph{and} $\f\!\star(-a)=-a$
    are $a = \<\0,\0\>$, $\<\1,\0\>$, $\<\0,\1\>$, or $\<\1,\1\>$.
    Further, if $\f\!\star(a)\neq\<\1,\1\>$ \emph{and} 
    $\f\!\star(-a)\neq\<\1,\1\>$, then $a$ must be either $\<\1,\0\>$
    or $\<\0,\1\>$. 
    It follows that for any identity $e$, $\gA\models \forall \vec x\ e(\vec x)$
    if and only if
    \[ \gA\star\models \forall \vec x\forall y\big( 
    \f(y)=\1 \;\lor\; f(-y)=\1 \;\lor\; ( \f(y)=y\land \f(-y)=-y\land
    e^y(\vec x)) \big),\label{eq:formula} \]
    cf. \eqref{eq:erelativized}. Note that the formula in \eqref{eq:formula}
    (which we will denote by $\phi_e(\vec x, y)$) is a positive universal formula.

    As $\VAR$ and $\UAR$ are different varieties, there is an identity $e(\vec x)$
    such that 
    \[\VAR\models\forall x\  e(\vec x)\qquad \text{ while }\qquad  \UAR\nmodels\forall x \ e(\vec x)\,, \label{eq:anemb} \]
    or the other way around -- let us suppose \eqref{eq:anemb}. Let 
    $\gB\in\UAR$ be such that $\gB\nmodels\forall x \ e(\vec x)$.
    Using the argument above we obtain
    \[(\forall\gA\in\VAR)\ \gA\star\models \phi_e(\vec x, y)\qquad \text{ while }\qquad  \gB\star\nmodels\phi_e(\vec x,y)\,. \]
    In congruence distributive varieties, J\'onsson's lemma \cite{Jonsson1967} states that the subdirectly irreducible members $(\mathbf{HSP}(K))_{SI}$ of the variety $\mathbf{HSP}(K)$ belong to $\mathbf{HSP}_U(K)$. In particular,
    \[ (\VAR\star)_{SI} \ \subseteq\ \mathbf{HSP}_U\big(\{\gA\star:\; \gA\in\VAR\}\big).\label{eqanemb2}
    \]
    The operations $\mathbf{H}$, $\mathbf{S}$ and $\mathbf{P}_U$ preserve
    positive universal formulas. Therefore, by \eqref{eqanemb2}, every member of 
    $(\VAR\star)_{SI}$ must satisfy $\phi_e(\vec x,y)$, while on the other
    hand, $\gB\star\nmodels\phi_e(\vec x,y)$. But $\gB\star$ is simple, by Lemma \ref{lem:simple}, hence $\gB\star$ cannot belong to $\VAR\star$. This proves
    that $\VAR\star\neq\UAR\star$.
\end{proof}

\begin{remark}\label{remark:tab}
    In Theorem \ref{thm:main1} if $\VAR$ is generated by the class $\mathsf{K}$ of algebras, then $\VAR\star$ can be taken to be 
    \[
        \VAR\star = \HSP\{\gA\star:\; \gA\in \mathsf{K} \}\,.
    \]
    The construction $\gA\mapsto\gA\star$ preserves the finiteness of the algebra. Thus, if $\VAR$ is generated by a finite set of finite algebras, then so is $\VAR\star$. In particular, if $\VAR$ is tabular (generated by a single finite algebra), then so is $\VAR\star$.
\end{remark}

\begin{theorem}\label{thm:main2}
    For each variety $\VAR$ of \emph{normal} and \emph{unit-preserving} Boolean frames
    there exists a variety $\VAR\star$ of normal and unit-preserving Boolean frames
    such that the following properties hold.
    \begin{itemize}\itemsep-2pt
        \item $\VAR\star$ lacks the congruence extension property.
        \item If $\VAR\neq\UAR$, then $\VAR\star\neq\UAR\star$.
        \item The construction $\VAR\mapsto\VAR\star$ preserves the properties of being (not) extensive, (not) monotone, (not) idempotent.\footnote{E.g. if $\VAR$ is monotone but not extensive, then so is $\VAR\star$.}
    \end{itemize}
\end{theorem}
\begin{proof}
    Items (i) and (ii) are immediate from Theorem \ref{thm:main1}. As for (iii) one only needs to check the definition of $\gA\star$ in Construction \ref{constr:acsillag}. It is routine to show that $\gA\star$ is (not) extensive / (not) monotone / (not) idempotent if $\gA$ is so.
\end{proof}


\subsection{The semi-complemented case}\label{subsec:sc}

In this subsection, we construct continuum many normal, unit-preserving, and semi-complemented varieties of Boolean frames that lack the congruence extension property. 

\begin{construction}
    Let $\gA = \<A, \LAND, -, \0, \f\>$ be a \emph{normal} and \emph{unit-preserving} Boolean frame with at least $8$ elements. We construct the Boolean
    frame $\gA^\sharp$ as follows.
    \begin{itemize}\itemsep-2pt
        \item The universe of $\gA^\sharp$ is $A^\sharp\defeq A\times A$.
        \item The Boolean reduct of $\gA^\sharp$ is the direct product
        $\Bl\gA\times\Bl\gA$ of the Boolean reduct of $\gA$.
        \item The operation $\f^\sharp: A^\sharp\to A^\sharp$ is defined as
        \begin{align}
            \f\nolimits^\sharp: \begin{cases}
                \<a, \0\> &\mapsto\quad \< f(a), \0\> \\
                \<\0, a\> &\mapsto\quad \< \0, f(a) \>\\
                \<a, \1\> &\mapsto\quad \<-f(-a), \1\> \\
                \<\1, a\> &\mapsto\quad \<\1, -f(-a)\> 
            \end{cases}
        \end{align}
        and for every other $\<a,b\>$ not listed above, i.e. when $\0<a<\1$ and
        $\0<b<\1$, $\f^\sharp$ maps
        $\<a,b\>$ either to $\<\0,\0\>$ or to $\<\1,\1\>$ subject to two
        conditions:
        \begin{enumerate}[(i)]\itemsep-2pt
            \item $\f^\sharp(\<-a,-b\>) = -\f^\sharp(\<a,b\>)$, and
            \item for every $\0<x<\1$ there are $\0<y_1, y_2<\1$ such that
            \[
                \f\nolimits^\sharp(\<x,y_1\>) = \<\0,\0\>,\quad\text{ and }\quad
                \f\nolimits^\sharp(\<x,y_2\>) = \<\1,\1\>,
            \]
            and symmetrically, for every $\0<y<\1$ there are $\0<x_1, x_2<\1$
            such that 
            \[
                \f\nolimits^\sharp(\<x_1,y\>) = \<\0,\0\>,\quad\text{ and }\quad
                \f\nolimits^\sharp(\<x_2,y\>) = \<\1,\1\>,
            \]
        \end{enumerate}
        As $A$ is large enough (has at least $8$ elements), conditions (i) and (ii) can be satisfied.
    \end{itemize}
    It is easy to check that $\gA^\sharp$ is \emph{semi-complemented}, normal, and unit-preserving. To simplify notation we write $f^\sharp(a,b)$ in place of $f^\sharp(\<a,b\>)$.
\end{construction}

\begin{lemma}\label{lem:sc1}
    Assume that $\gA$ is normal and unit-preserving and for every $a\neq\0$
    there is a natural number $k$ such that $f^k(a)=\1$.
    Then 
    $\gA^\sharp$ is simple and does not have the congruence extension property.
\end{lemma}
\begin{proof}
    Let $\sim$ be not the least congruence of $\gA^\sharp$, and assume that there are
    $x,y\in A$ such that $\<x,y\>\neq\<\0,\0\>$ but $\<x,y\>\sim\<\0,\0\>$. 
    Then $x\neq 0$ or $y\neq 0$, let us say $x\neq 0$. Then
    $\<x,\0\>\sim\<\0,\0\>$ as well. By construction of $\gA^\sharp$ there
    are $\0<y_1, y_2<\1$ such that
    \[
      \f\nolimits^\sharp(x,y_1) = \<\0,\0\>,\quad\text{ and }\quad
      \f\nolimits^\sharp(x,y_2) = \<\1,\1\>.
    \]
    As $\<0,y_1\>\sim\<x,y_1\>$ and $\<0,y_2\>\sim\<x,y_2\>$, there is $k\in\omega$ such that
    \begin{align}
    \<\0,\0\>\ =\ (\f\nolimits^\sharp)^k(x,y_1)&\ \sim\ (\f\nolimits^\sharp)^k(0,y_1)\ =\ \<0,\1\>,
    \quad\text{and }\\
    \<\1,\1\>\ =\ (\f\nolimits^\sharp)^k(x,y_2)&\ \sim\ (\f\nolimits^\sharp)^k(0,y_2)\  = \ \<0,\1\>.
    \end{align}
    Thus $\<\0,\0\>\sim\<\1,\1\>$, proving that $\sim$ is the largest congruence. Consequently, $\gA^\sharp$ is simple. 

    As for the lack of the congruence extension property, notice that
    \[
        \big\{ \<\0,\0\>, \<\0,\1\>, \<\1,\0\>, \<\1,\1\> \big\}
    \]
    is the universe of a subalgebra of $\gA^\sharp$ which has four congruences
    (cf. Figure \ref{figure:kisba}). The non-trivial congruences of this
    subalgebra extend to trivial congruences of $\gA^\sharp$.
\end{proof}

\noindent Recall from \eqref{relterm} that for a term $t(\vec x)$ and a variable $y$ not occurring in $\vec x$, we defined the relativized term $t^y(\vec x)$, and similarly for an identity $t_1=t_2$ we let $(t_1=t_2)^y \defeq t_1^y=t_2^y$.

\begin{lemma}\label{lem:sc2}
    Assume that $\gA$ is normal and unit-preserving, and there is a natural number
    $k$ such that for every $a\neq \0$ we have $\f^k(a)=\1$. Let $e(\vec x)$ be an identity. The following are equivalent:
    \begin{enumerate}[(i)]\itemsep-1pt
        \item $\gA\models \forall\vec x\;e(\vec x)$.
        \item $\gA^\sharp\models \forall \vec x\forall y\big(
            \f^k(y)=\1\ \lor\  \f^k(y)=\0\ \lor\  e^{\f^k(y)}(\vec x)
        \big)$.
    \end{enumerate}
\end{lemma}
\begin{proof}
    For every element $y$ of $\gA^\sharp$, $(\f\nolimits^\sharp)^k(y)$ is one
    of $\<\0,\0\>$, $\<\0,\1\>$, $\<\1,\0\>$, or $\<\1,\1\>$, and each of these are values of $(\f^\sharp)^k(y)$ for some $y$. To complete the proof observe that 
    \[
        \gA\models t_1=t_2 \quad \text{ if and only if }\quad 
        \gA^\sharp\models t_1^{\<\1,\0\>} = t_2^{\<\1,\0\>}\,,    
    \]
    for any identity $t_1=t_2$ (and similarly with $\<\0,\1\>$ in place of $\<\1,\0\>$).
\end{proof}

\noindent We recall the definition of a wheel frame from \cite{Miyazaki2005}. For $n\geq 5$ the wheel frame $\cW_n = \<W_n, R_n\>$ is the frame
\begin{align}
    W_n&\defeq\{ 0,\ldots, n-1 \}\cup\{h\} \\
    R_n&\defeq\{ \<x,y\>: x, y<n, |x-y|\leq 1\ (\text{mod } n)\}\cup\{ \<h,h\>, \<h,x\>, \<x,h\>: x<n\}.
\end{align}
    \begin{figure}[!ht]
    \begin{center}
        \begin{tikzpicture}[thick,scale=1]
            \graph [nodes={draw, circle, fill}, clockwise, radius=1cm, empty nodes]
  {subgraph A[at={(0,-1)}] -- subgraph C_n[n=9]};
        \end{tikzpicture}
    \end{center}
    \caption{The frame $\cW_9$.}
    \end{figure}
For $n\geq 5$ let $\gW_n$ be the complex algebra of $\cW_n$. Then $\gW_n$ is a normal, unit-preserving Boolean frame, having at least $8$ elements, and for $a\in \gW_n$, $a\neq \0$ we have $\f(\f(a))=\1$. Let $Prim$ be the set of prime numbers larger than $4$. According to \cite[Theorem 21]{Miyazaki2005}, for different $X, Y\subseteq Prim$, we have
\[
    \HSP\{ \gW_n:\; n\in X\} \neq \HSP\{ \gW_n:\; n\in Y\}\,.
\]
For $X\subseteq Prim$ let us write
\[
    \VAR(X)&\defeq\HSP\{ \gW_n:\; n\in X\} \\
    \VAR^\sharp(X) &= \HSP\{\gW_n^\sharp:\; n\in X\}.\label{eq:ssdd}
\]
By Lemma \ref{lem:sc1}, each $\gW_n^\sharp$ is simple and does not have the congruence extension property. Therefore, $\VAR^\sharp(X)$ is a variety of normal, unit-preserving, and semi-complemented Boolean frames lacking the congruence extension property. \medskip

\begin{theorem}\label{THM:COMP}
    For distinct $X,Y\subseteq Prim$ we have 
    $\VAR^\sharp(X)\neq \VAR^\sharp(Y)$. In particular, 
    there are continuum many varieties of normal, unit-preserving, and semi-complemented Boolean frames lacking the congruence extension property. 
\end{theorem}
\begin{proof}
    $X\neq Y$ implies $\VAR(X)\neq\VAR(Y)$, and thus there is an identity 
    $e(\vec x)$ such that 
    \[\VAR(X)\models\forall x\  e(\vec x)\qquad \text{ while }\qquad  \VAR(Y)\nmodels\forall x \ e(\vec x)\,, \label{eq:anemb2} \]
    or the other way around -- let us suppose \eqref{eq:anemb2}. 
    Let $n\in Y$ be such that $\gW_n\nmodels\forall x \ e(\vec x)$.
    Let $\psi_e(\vec x, y)$ be the formula 
    \[
        \forall \vec x\forall y\big(
            \f(\f(y))=\1\ \lor\  \f(\f(y))=\0\ \lor\  e^{\f(\f(y))}(\vec x)\big)\,.
    \]
    By Lemma \ref{lem:sc2} for each $k\in\omega$ we have
    \[
        \gW_k\models \forall\vec x\;e(\vec x)\quad\text{ iff }\quad
        \gW_k^\sharp\models \psi_e(\vec x, y)\,.
    \]
    We then have 
    \[
        (\forall k\in X)\ \gW_k^\sharp\models \psi_e(\vec x, y)\qquad \text{ while }\qquad  \gW_n^\sharp\nmodels\psi_e(\vec x,y)\,. \label{eqanemb44}
    \]
    By J\'onsson's lemma \cite{Jonsson1967}, 
    \[ (\VAR^\sharp(X))_{SI} \ \subseteq\ \mathbf{HSP}_U\big(\{\gW_k^\sharp:\; k\in X\}\big).
    \]
    As operations $\mathbf{H}$, $\mathbf{S}$ and $\mathbf{P}_U$ preserve
    positive universal formulas, and $\psi_e(\vec x, y)$ is a positive universal formula, by \eqref{eqanemb44} it follows that every member of 
    $(\VAR^\sharp(X))_{SI}$ must satisfy $\psi_e(\vec x,y)$. By lemma \ref{lem:sc1}, 
    $\gW_n^\sharp$ is simple, and by \eqref{eqanemb44}, $\gW_n^\sharp\nmodels\psi_e(\vec x,y)$, consequently 
    $\gW_n\notin (\VAR^\sharp(X))_{SI}$. This proves that $\VAR^\sharp(X)\neq \VAR^\sharp(Y)$.
\end{proof}


\subsection{The symmetric and extensive case}\label{subsec:symext}
We say that a Boolean frame $\gA$ is symmetric if $x\leq -f(-f(x))$ holds for $\gA$. In this subsection, we construct continuum many normal, unit-preserving, and symmetric varieties of Boolean frames that lack the congruence extension property. The construction is done by mixing the techniques from the $(\cdot)^*$ and the $(\cdot)^\sharp$ constructions from Subsections \ref{subsec:main},\ref{subsec:sc}.

\begin{construction}
	For an arbitrary Boolean frame $\gA = \<A, \LAND, -, \0, \f\>$
	we construct the Boolean frame $\gA^\flat$ as follows. 
	\begin{itemize}\itemsep-2pt
		\item The universe of $\gA^\flat$ is $A^\flat\defeq A\times A$.
		\item The Boolean reduct of $\gA^\flat$ is the direct product $\Bl\gA\times\Bl\gA$ of the Boolean reduct of $\gA$.
		\item The operation $\f\!^\flat:A\times A\to A\times A$ is defined as
	\begin{align}
		\f\nolimits^\flat: \begin{cases}
			\<a, \0\> &\mapsto\quad \< f(a), \0\> \\
			\<\0, a\> &\mapsto\quad \< \0, f(a) \>\\
			\<a, \1\> &\mapsto\quad \<f(a), \1\> \\
			\<\1, a\> &\mapsto\quad \<\1, f(a)\>\\
			 	\<b, c\> &\mapsto\quad \<\1, 1\>\\
		\end{cases}
	\end{align}
for $0<b,c<1$.
	\end{itemize}
\end{construction}
\noindent It is routine to check that $\gA^\flat$ is symmetric if $\gA$ is and $\gA^\flat$ is extensive if $\gA$ is.  By ($\flat$) let us abbreviate the following conditions on a Boolean frame $\gA$:
\begin{itemize}\itemsep-2pt
	\item $\gA$ is normal and unit-preserving,
	\item there are $0<a,b\in A$ such that $f(a)\LAND f(b) = 0$,
 \item $f(c)\neq 0$ for all $0\neq c\in A$.
\end{itemize}

\begin{lemma}\label{lem:simpl_flat}
	Let $\gA= \<A, \LAND, -, \0, \f\>$ be a Boolean frame satisfying  $(\flat)$.  Then $\gA^\flat$ is simple and does not have the congruence extension property.
\end{lemma}
\begin{proof}
	The proof follows the pattern of Lemma \ref{lem:simple}. 
    Take any non-trivial congruence $\sim$ of $\gA^\flat$. Then there is
	$\<x,y\>\in A^\flat$ such that $\<x,y\>\neq\<\0,\0\>$ but 
	$\<x,y\>\sim\<\0,\0\>$. Without loss of generality we may assume $x\neq\0$, and thus $\<x,\0\>\sim\<\0,\0\>$. Let $a,b\in A$ be from the condition ($\flat$). Observe that $a,  b<1$. Otherwise, if, say $a=1$, then $f(a)\LAND f(b)= 1\LAND f(b) =0$ contradicting to $f(b)\neq 0$. Considering 
	$\<x,a\>$ and $\<x,b\>$, since $\<x,a\>\sim\<\0,a\>$ and $\<x,b\>\sim\<\0,b\>$ we  get
	\[
	\<\1,\1\> \ =\  \f\!^\flat(\<x,a\>)\ \sim\ \f\!^\flat(\<\0,a\>)\  =\  \<\0, \f(a)\>,\\
	\<\1,\1\> \ =\  \f\!^\flat(\<x,b\>)\ \sim\ \f\!^\flat(\<\0,b\>) \ =\  \<\0, \f(b)\>.
	\]
	The rest is similar to that of  Lemma \ref{lem:simple} and Lemma \ref{lem:nocep}.
\end{proof}

\noindent Just as in  \eqref{eq:erelativized}, for any \emph{normal} Boolean frame $\gA$ we have
\[
\gA\models t_1=t_2 \quad \text{ if and only if }\quad 
\gA^\flat\models t_1^{\<\1,\0\>} = t_2^{\<\1,\0\>}\,,    \label{eq:flat_rel}
\]
for any identity $t_1=t_2$, since $f^*(\<1,0\>)=f^\flat(\<1,0\>)$ (and similarly  with  $\<\0,\1\>$ in place of $\<\1,\0\>$).\medskip

\noindent Similarly to \eqref{eq:ssdd} for $X\subseteq Prim$ we let
\[
\VAR^\flat(X) &= \HSP\{(\gW_n\times \gW_n)^\flat  :\; n\in X\}.
\]
 Since the complex algebra $(\cW_n\uplus \cW_n)^+\cong \gW_n\times \gW_n $, it is easy to see that $(\flat)$ holds for $\gW_n\times \gW_n$.    Hence, by  Lemma \ref{lem:simpl_flat}  $(\gW_n\times \gW_n)^\flat$ is  simple and does not have the congruence extension property. Therefore, $\VAR^\flat(X)$ is a variety of normal, unit-preserving, extensive, and symmetric Boolean frames lacking the congruence extension property. Combining the  techniques from Theorem \ref{THM:COMP} and Theorem \ref{thm:main1} we have the following. \medskip

\begin{theorem}\label{thm:main3}
	For distinct $X,Y\subseteq Prim$ we have 
	$\VAR^\flat(X)\neq \VAR^\flat(Y)$. In particular, 
	there are continuum many varieties of normal, unit-preserving, extensive, and symmetric Boolean frames lacking the congruence extension property. 
\end{theorem}
\begin{proof}
	$X\neq Y$ implies $\VAR(X)\neq\VAR(Y)$, and thus there is an identity 
	$e(\vec x)$ such that 
	\[\VAR(X)\models\forall x\  e(\vec x)\qquad \text{ while }\qquad  \VAR(Y)\nmodels\forall x \ e(\vec x)\,, \label{eq:anemb2-c} \]
	or the other way around -- let us suppose \eqref{eq:anemb2-c}.  
	Let $n\in Y$ be such that $\gW_n\nmodels\forall x \ e(\vec x)$. Then $\gW_n\times \gW_n\nmodels\forall x \ e(\vec x)$, since $\gW_n$ diagonally embeds into $\gW_n\times \gW_n$.

    For every $(\gW_m\times \gW_m)^\flat$
    the only elements $a$ such that 
    $\f^\flat(\f^\flat(a)) = a$ \emph{and} $\f^\flat(\f^\flat(-a))=-a$ hold are $a = \<\0,\0\>$, $\<\1,\0\>$, $\<\0,\1\>$, or $\<\1,\1\>$.
    Further, if $\f^\flat(\f^\flat(a))\neq\<\1,\1\>$ \emph{and} 
    $\f^\flat(\f^\flat((-a))\neq\<\1,\1\>$, then $a$ must be either $\<\1,\0\>$
    or $\<\0,\1\>$. 
    Using (\ref{eq:flat_rel})  it follows that for any identity $e$, $\gW_m\times \gW_m\models \forall \vec x\ e(\vec x)$
    if and only if
	\[(\gW_m\times \gW_m)^\flat\models
	\forall \vec x\forall y\big(
	\f(\f(y))=\1\ \lor\  \f(\f(y))=\0\ \lor\  e^{\f(\f(y))}(\vec x)\big)\,.
	\]
	Let this formula be $\psi_e(\vec x,y)$. Then we have
	\[
	\gW_m\times \gW_m\models \forall\vec x\;e(\vec x)\quad\text{ iff }\quad
	(\gW_m\times \gW_m)^\flat\models \psi_e(\vec x, y)\,.
	\]
	for all $m\in Prim$.  Therefore
	\[
	(\forall k\in X)\ (\gW_k\times\gW_k)^\flat\models \psi_e(\vec x, y)\qquad \text{ while }\qquad (\gW_n\times\gW_n)^\flat\nmodels\psi_e(\vec x,y)\,. \label{eqanemb_flat}
	\]
	By J\'onsson's lemma \cite{Jonsson1967}, 
	\[ (\VAR^\flat(X))_{SI} \ \subseteq\ \mathbf{HSP}_U\big(\{(\gW_k\times\gW_k)^\flat:\; k\in X\}\big).
	\]
	As operations $\mathbf{H}$, $\mathbf{S}$ and $\mathbf{P}_U$ preserve
	positive universal formulas, and $\psi_e(\vec x, y)$ is a positive universal formula, by \eqref{eqanemb_flat} it follows that every member of 
	$(\VAR^\flat(X))_{SI}$ must satisfy $\psi_e(\vec x,y)$. By Lemma \ref{lem:simpl_flat}, 
	$(\gW_n\times\gW_n)^\flat$ is simple, and by \eqref{eqanemb_flat}, $(\gW_n\times\gW_n)^\flat\nmodels\psi_e(\vec x,y)$, consequently 
	$(\gW_n\times\gW_n)^\flat\notin (\VAR^\flat(X))_{SI}$. This proves that $\VAR^\flat(X)\neq \VAR^\flat(Y)$.
\end{proof}



\section*{Funding}
The first author was supported by the grant 2019/34/E/HS1/00044 of the National Science Centre (Poland), and by the grant of the Hungarian National Research, Development and Innovation Office, contract number: K-134275.
The second author is supported by the ÚNKP-23-3 New National Excellence Program of the Ministry for Culture and Innovation from the source of the national research, development and innovation fund.

\section*{Acknowledgement}
The authors wish to acknowledge the Cracow Logic Conference (CLoCk) 2023 where the open problems were presented. Special thanks goes to Krzysztof Krawczyk.

\section*{Appendix}

That every axiomatic extension of $\proves_{\bfK}$ has LDDT is a well-known theorem (cf. \cite{Krawczyk2023} or \cite{Czelakowski1986}). The algebraic formulation of this theorem is that any normal and additive Boolean frame has the congruence extension property (cf. \cite{Kowalski2000}). The proof of this statement is rather simple, but for completeness, we provide a proof below. Then, we show that normality is not needed, i.e. that additive, but not necessarily normal Boolean frames have CEP as well.

\begin{proposition}\label{prop:nacep}
    If $\gA$ is a  normal and additive Boolean frame, then it has the CEP.
\end{proposition}
\begin{proof}
    Let $\gB$ be a subalgebra of $\gA$ and $\Theta\in\Con(\gB)$. 
    The set
    \[ F = \{ x\leftrightarrow y:\; x\;\Theta\;y,\; x,y\in B\}\subseteq B\]
    is a normal, congruential filter: if $x\in F$ then $f(x)\in F$, and
    if $x\leftrightarrow y\in F$, then
    $f(x)\leftrightarrow f(y)\in F$. (Cf. \cite[Sec. 5]{Krawczyk2023} or \cite[p.223]{chagrov1997modal}). Let us define 
    \[ G = \{a\in A:\; \exists b\in F\; ( b\leq a)\}\,. \]
    Then $G$ is a filter on $\gA$ which extends $F$. We show that
    $G$ is normal and congruential. Pick $a\leftrightarrow b\in G$. Then
    there is $x\in F$ such that $x$ $\leq$ $a\leftrightarrow b$. 
    By monotonicity of $f$ we get
    \[
        f(x) \leq f(a\leftrightarrow b)\,,
    \]
    and by additivity
    \[
        f(a\leftrightarrow b) \leq f(a)\leftrightarrow f(b)\,.
    \]
    As $f(x)\in F$, we get $f(a)\leftrightarrow f(b)\in G$. That $G$ is normal follows from monotonicity of $f$. 

    The congruence
    \[ \Psi = \{  (a,b):\; a,b\in A,\; a\leftrightarrow b\in G\}\]
    corresponding to $G$ is the desired extension of $\Theta$.
\end{proof}

\begin{proposition}
    If $\gA$ is an additive (but not necessarily normal) Boolean frame, then it has the CEP.
\end{proposition}
\begin{proof}
    Suppose $\gA = \<A, f\>$ is a Boolean frame such that $f$ is additive. Note that additivity implies monotonicity, therefore
    $f(\0)\leq f(x)$ for every $x\in A$. Let us define the operation $g:A\to A$ by
    \[
        g(x) \defeq f(x)- f(\0)\,.
    \]
    It is straightforward to check that $g$ is normal and additive. Indeed, $g(\0) = f(\0)- f(\0) = \0$, and 
    \[
       g(x\LOR y) \quad&=\quad f(x\LOR y)- f(\0) \quad=\quad
       \big( f(x)\LOR f(y)\big)- f(\0) \\ &=\quad 
       (f(x)- f(\0))\ \LOR\  (f(y)- f(\0)) \\ &=\quad 
       g(x)\LOR g(y).
    \]
    As $f(\0)\leq f(x)$ we have 
    $f(x) = (f(x)-f(\0))\LOR f(\0)$. Thus
    \[ f(x) = g(x)\LOR f(\0) \qquad\text{for all } x\in A\,. \]
    Next, we show 
    \[ \Con(A, f) = \Con(A, g)\,\]
    For $\Theta\in \Con(A,f)$ we have 
    \[  x\;\Theta\; y\quad&\Rightarrow\quad f(x)\;\Theta\; f(y) \\
        &\Rightarrow\quad (f(x)-f(\0)) \;\Theta\; (f(y)-f(\0)) \\
        &\Rightarrow\quad g(x) \;\Theta\; g(y)\,, \]
    hence $\Theta\in\Con(A,g)$. Similarly, for $\Theta\in\Con(A,g)$
    we have
    \[  x\;\Theta\; y\quad&\Rightarrow\quad g(x)\;\Theta\; g(y) \\
        &\Rightarrow\quad (g(x)\LOR f(\0)) \;\Theta\; (g(y)\LOR f(\0)) \\
        &\Rightarrow\quad f(x) \;\Theta\; f(y)\,, \]
    hence $\Theta\in\Con(A,f)$.

    Take a subalgebra $(B,f)$ of $(A,f)$ and a congruence $\Psi\in\Con(B,f)$. Then $(B,g)$ is a subalgebra of $(A,g)$, and $\Psi\in\Con(B,g)$. As $(B,g)$ is normal and additive, by Proposition \ref{prop:nacep} it has the congruence extension property, and thus there is a congruence $\Theta\in \Con(A,g)$ such that $\Theta\cap (B\times B) = \Psi$. But $\Theta\in\Con(A,f)$ as well.    
\end{proof}

Finally, we note that additivity is not necessary for having the CEP. In fact, it is straightforward to construct not additive varieties of Boolean frames having the congruence extension property. For instance, let $\VAR$ be the variety of Boolean frames satisfying the identity $\f(x)=-x$. As $\f$ is Boole-definable, the variety inherits the CEP from Boolean algebras.

\begin{thebibliography}{10}
\providecommand{\selectlanguage}[1]{\relax}

\bibitem{BP89}
W.~J. Blok, D.~Pigozzi, \emph{Algebraizable logics}, \textbf{Mem. Amer. Math.
  Soc.}, vol.~77(396) (1989), pp. vi+78.

\bibitem{chagrov1997modal}
A.~Chagrov, M.~Zakharyaschev, \textbf{Modal Logic}, Oxford Logic Guides,
  Clarendon Press (1997).

\bibitem{Czelakowski1985}
J.~Czelakowski, \emph{Algebraic Aspects of Deduction Theorems}, \textbf{Studia
  Logica}, vol.~44(4) (1985), pp. 369--387.

\bibitem{Czelakowski1986}
J.~Czelakowski, \emph{Local Deductions Theorems}, \textbf{Studia Logica},
  vol.~45(4) (1986), pp. 377--391.

\bibitem{Fine1974}
K.~Fine, \emph{An Ascending Chain of {S}4 Logics}, \textbf{Theoria}, vol.~40(2)
  (1974), pp. 110--116.

\bibitem{Hansson1973}
B.~Hansson, P.~G\"{a}rdenfors, \emph{A Guide to Intensional Semantics}, [in:]
  S.~Halld\'{e}n (ed.), \textbf{Modality, Morality and Other Problems of Sense
  and Nonsense}, Lund, Gleerup (1973), pp. 151--167.

\bibitem{Jonsson1967}
B.~J\'onsson, \emph{Algebras whose congruence lattices are distributive},
  \textbf{Mathematica Scandinavica}, vol.~21(2) (1967).

\bibitem{Kowalski2000}
T.~Kowalski, \emph{A remark on quasivarieties of modal algebras},
  \textbf{Bulletin of the Section of Logic}, vol.~29(1) (2000).

\bibitem{Krawczyk2023}
K.~A. Krawczyk, \emph{Deduction Theorem in Congruential Modal Logics},
  \textbf{Notre Dame Journal of Formal Logic}, vol.~64(2) (2023), pp. 185--196.

\bibitem{Miyazaki2005}
Y.~Miyazaki, \emph{Normal Modal Logics Containing KTB with Some Finiteness
  Conditions}, [in:] R.~Schmidt, I.~Pratt{-}Hartmann, M.~Reynolds, H.~Wansing
  (eds.), \textbf{Advances in Modal Logic, Volume 5}, CSLI Publications (2005),
  pp. 171--190.

\bibitem{Ono2019}
H.~Ono, \textbf{Proof Theory and Algebra in Logic}, Springer Singapore,
  Singapore (2019).

\bibitem{Pacuit2017}
E.~Pacuit, \textbf{Neighborhood Semantics for Modal Logic}, Springer, Cham,
  Switzerland (2017).

\end{thebibliography}
\end{document}